\documentclass[12pt,a4paper]{article}

\usepackage[american]{babel}
\usepackage{amsmath,amsthm}
\usepackage{latexsym}
\usepackage{amssymb}
\usepackage[shortlabels]{enumitem}
\usepackage{xcolor}
\usepackage{tikz}
\usepackage{tikz-cd}
\usepackage{comment}

\renewcommand{\phi}{\varphi}
\newcommand{\RR}{\mathbb{R}}
\newcommand{\ZZ}{\mathbb{Z}}

\newcommand{\QQ}{\mathbb{Q}}
\newcommand{\CC}{\mathbb{C}}
\newcommand{\bS}{\mathbb{S}}

\newcommand{\bA}{\mathbf{A}}
\newcommand{\id}{\mathrm{id}}

\newcommand{\eps}{\varepsilon}
\newcommand{\Iso}{\mathrm{Iso}}
\newcommand{\Aut}{\mathrm{Aut}}
\newcommand{\Res}{\mathrm{Res}}
\newcommand{\IE}{\mathrm{IE}}

\newcommand{\cB}{\mathcal{B}}

\newcommand{\cS}{\mathcal{S}}

\newcommand{\cH}{\mathcal{H}}
\newcommand{\pr}{\mathrm{pr}}

\renewcommand{\emptyset}{\varnothing}
\renewcommand{\setminus}{-}

\theoremstyle{plain}
\newtheorem{Thm}[subsection]{Theorem}
\newtheorem{Cor}[subsection]{Corollary}
\newtheorem{Prop}[subsection]{Proposition}
\newtheorem{Lem}[subsection]{Lemma}

\theoremstyle{definition}

\newtheorem{Facts}[subsection]{Facts}

\newtheorem{Ex}[subsection]{Example}

\newtheorem{Def}[subsection]{Definition}

\allowdisplaybreaks

\title{\bf Some remarks on proper actions, proper metric spaces, and buildings}
\author{Linus Kramer\thanks{Funded by the Deutsche Forschungsgemeinschaft through a Polish-German
\emph{Beethoven} grant KR1668/11, and under
Germany's Excellence Strategy EXC 2044-390685587, Mathematics M\"unster: Dynamics-Geometry-Structure.}
}

\begin{document}

\date{\em Dedicated to the memory of Helmut Reiner Salzmann}

\maketitle

\begin{abstract}
We discuss various aspects of isometric group actions on proper metric spaces.
As one application, we show that a proper and Weyl transitive action on a 
euclidean building is strongly transitive on the maximal atlas 
(the complete apartment system) of the building.
\end{abstract}

The purpose of this article is to supply references and complete proofs
for some useful facts about proper group actions on proper metric spaces.
These results are then applied to group actions on metric simplicial complexes,
and in particular to group actions on buildings.
Some of these results are known as folklore, though it is
in some cases not so easy to find references or reliable proofs.

The first section contains a discussion of proper maps and proper actions
of topological groups in general. The second section discusses isometry groups of
metric spaces, and in particular isometry groups of proper metric spaces.
The third section applies these results to $M_\kappa$-simplicial complexes, which have
become a central tool in geometric group theory.
The fourth section introduces buildings (in their various manifestations).
One (reassuring) result for buildings is that all reasonable topologies
on their automorphism groups coincide.
The notion of a proper action of a topological group on a building is, however, more delicate.
The fifth an final section contains result about various types of 
transitive actions on buildings. For the case of proper actions, these 
transitivity conditions
are shown to be essentially equivalent.

All topological spaces and groups are assumed to be Hausdorff spaces.

\section{Proper maps and proper actions}
A continuous map $f:X\longrightarrow Y$ is called \emph{proper}
\footnote{This is Bourbaki's notion of properness and the reader should be warned
that other, different notions of proper maps appear in the literature.}
if it is closed and if preimages of points are compact \cite[I.\S10.2]{BourTop}.
Then the preimage of every compact subset is compact.
Proper maps are also called \emph{perfect maps} \cite[\S3.7]{Eng}.
If $Y$ is a \emph{$k$-space} (a subset $B\subseteq Y$ is closed if and only if
$B\cap K$ is compact for every compact subset $K\subseteq Y$),
then a continuous map $f:X\longrightarrow Y$ is proper if and only if the preimage
of every compact subset $B\subseteq Y$ is compact \cite[3.7.18]{Eng}.
We note that every first countable space, every locally compact space  and every CW complex is
a $k$-space, cp.~\cite[XI.9.3]{Dug} and \cite[App.]{Hatcher}. A product of $k$-spaces is not necessarily a $k$-space,
but Whitehead's Lemma assures that the product of a $k$-space and a locally compact space is again a $k$-space
\cite[XII.4.4]{Dug}.
We collect some facts about proper maps which can be found in Bourbaki \cite{BourTop}
and Engelking \cite{Eng}.

\begin{Facts}
\label{ProperFacts}
If $Y$ is locally compact and if ${f:X\longrightarrow Y}$ is
proper, then $X$ is also locally compact \cite[3.7.24]{Eng}.
If $f:X\longrightarrow Y$ is a proper map, if $A\subseteq X$ is
a closed subset, if $B\subseteq Y$ is any subset, and if $f(A)\subseteq B$,
then the restriction-corestriction $f|_A^B:A\longrightarrow B$ is also proper.
Similarly, the restriction-corestriction ${f|_{f^{-1}(B)}^B:f^{-1}(B)\longrightarrow B}$
is proper \cite[3.7.6]{Eng}.
The cartesian product of proper maps is again a proper map \cite[3.7.9]{Eng}.

Suppose that 
\[
\begin{tikzcd}
 X \arrow{r}{f}\arrow{dr}[swap]{h} & Y \arrow{d}{g} \\
 & Z
\end{tikzcd}          
             \]
is a commutative diagram of continuous maps.
If $f$ and $g$ are proper, then $h$ is also proper. If $h$ is proper,
then both $f$ and the restriction $g|_{f(X)}:f(X)\longrightarrow Z$ are proper \cite[3.7.3,3.7.5]{Eng}.
\end{Facts}
The following observation turns out to be useful.
\begin{Lem}\label{ProperLemmma}
Let $f:X\times Y\longrightarrow Z$ be a continuous map. 
For $C\times B\subseteq Z\times Y$, put 
\begin{align*}
X_{B,C} & =\{a\in X\mid\text{ there is $b\in B$ such that $f(a,b)\in C$}\} \\
&= \{a\in X\mid f(\{a\}\times B)\cap C\neq\emptyset\}.
 \end{align*}
If $Z\times Y$ is a $k$-space, then the following are 
equivalent.
\begin{enumerate}[\rm(i),nosep]
 \item The map $F:X\times Y\longrightarrow Z\times Y$ that maps $(x,y)$ to $(f(x,y),y)$ is proper.
 \item For all compact subsets $C\times B\subseteq Z\times Y$,
 the set $X_{B,C}$ is compact.
 \item For all compact subsets $C\times B\subseteq Z\times Y$,
 the set $X_{B,C}$ has compact closure in $X$.
\end{enumerate}
\end{Lem}
In general, (i) $\Longrightarrow$ (ii) $\Longrightarrow$ (ii) $\Longrightarrow$ (iii).
\begin{proof}
 We show first (i) $\Longrightarrow$ (ii). Let $C\times B$ be compact. Then the set 
 \[K=\{(x,y)\in X\times Y\mid f(x,y)\in C\times B\text{ and }y\in B\}=
  \{(a,b)\in X\times B\mid f(a,b)\in C\}
 \]
 is compact, and so is its projection $\pr_X(K)=X_{B,C}$. 
 
 Clearly, (ii) $\Longrightarrow$ (iii). 
 Now we show (iii) $\Longrightarrow$ (i), assuming that $Z\times Y$ is a $k$-space. 
 Let $L\subseteq Z\times Y$ be compact.
 We have to show that the preimage of $L$ in $X\times Y$ is compact.
 Put $B=\pr_Y(L)$ and $C=\pr_Z(L)$. Then $C\times B$ is compact and 
 contains $L$. The preimage of $C\times B$ in $X\times Y$ is closed and contained
 in compact set $\overline{X_{B,C}}\times B$. Therefore the preimage of $L$ is also compact.
\end{proof}
If in addition $Z=Y$ holds in Lemma~\ref{ProperLemmma}, then
it suffices to show that $X_{B,B}$ has compact closure for all compact sets $B\subseteq Y$,
because $B\times C\subseteq (B\cup C)\times (B\cup C)$.
\begin{Prop}\label{ProperProp}
Suppose that $F:X\times Y\longrightarrow Z\times Y$ is a proper map of the form 
$F(x,y)=(f(x,y),y)$. Then the following hold.
\begin{enumerate}[\rm(i),nosep]
 \item For each $y\in Y$, the set $X(y)=f(X\times\{y\})\subseteq Z$ is closed
 and the map $X\longrightarrow X(y)$ that maps $x$ to $f(x,y)$ is proper and in particular a quotient map.
 \item For all $(y,z)\in Y\times Z$, the set $X_{\{y\},\{z\}}=\{x\in X\mid f(x,y)=z\}$ is compact.
 \item Suppose that $Y=Z$, that the map $f$ is open and that $y\sim f(x,y)$
 is an equivalence relation on $Y$. Then the quotient map $q:Y\longrightarrow Y/{\sim}$ is open and 
 $Y/{\sim}$ is a Hausdorff space.
\end{enumerate}
\end{Prop}
\begin{proof}
 The restriction-corestriction $X\times\{y\}\longrightarrow X(y)\times\{y\}$
 is proper by \ref{ProperFacts}. In particular, it is a closed surjective map and (i) holds.
 For (ii) we just note that $X_{\{y\},\{z\}}=\pr_X(F^{-1}(z,y))$ is compact.
 For (iii), let  $U\subseteq Y$ be open. Then $q^{-1}(q(U))=f(X\times U)$
 is open, hence $q$ is open. 
 But then $q\times q$ is also open, and hence is a quotient map.
 The preimage of the diagonal $D$ in $Y/{\sim}\times Y/{\sim}$
 is $E=F(X\times Y)\subseteq Y\times Y$, which is closed.
 Hence $D$ is closed, and therefore $Y/{\sim}$ is Hausdorff.
\end{proof}
\begin{Def}
Suppose that $G$ is a topological group, that $X$ is a topological space, and that $G\times X\longrightarrow X$
is a continuous action. 
The kernel $N$ of a continuous action $G\times X\longrightarrow X$ is a closed normal subgroup, and the 
induced action $G/N\times X\longrightarrow X$ is again continuous, because the quotient map $G\longrightarrow G/N$ is open.
The action is called \emph{proper} 
\footnote{Again, the reader should be warned that other notions of proper
actions appear in the literature, see e.g.~\cite{Anton,Kapo}.}
if the map
\[
 G\times X\longrightarrow X\times X,\ (g,x)\longmapsto (gx,x)
\]
is proper. 
\end{Def}
\begin{Prop}\label{ProperActionProperties}
Let $G\times X\longrightarrow X$ be a continuous action of a topological group on a topological space $X$.
\begin{enumerate}[\rm(i),nosep]
 \item If $X\times X$ is a $k$-space (eg. if $X$ is first countable or locally compact) then the action is proper if and only if
 for all compact subsets $B,C\subseteq X$ the set $G_{B,C}=\{g\in G\mid g(B)\cap C\neq\emptyset\}$ is compact
 (or, equivalently, $G_{B,C}$ has compact closure in $G$).
 \item If the action is proper, then for every $x\in X$ the $G$-orbit $G(x)\subseteq X$ is closed, the stabilizer $G_x$ is compact,
 and the natural map $G/G_x\longrightarrow G(x)$ is a homeomorphism.
 The orbit space $G\backslash X$ is a Hausdorff space and the quotient map $q:X\longrightarrow G\backslash X$ is open.
\end{enumerate}
\end{Prop}
\begin{proof}
Claim (i) follows from Lemma~\ref{ProperLemmma} and Claim (ii) is a consequence of Proposition~\ref{ProperProp}.
See also \cite[III.\S4.2]{BourTop} for a direct proof.
\end{proof}
If $G$ acts properly on $X$, then the kernel $N$ of the action is a compact normal subgroup, and the 
induced action of $G/N$ on $X$ is also proper. Conversely, If $G$ acts continuously and with a compact
kernel $N$ on $X$, then the action of $G$ is proper if the induced action of $G/N$ is proper,
cp.~Lemma~\ref{ProperActionEquivariant} below.
Hence it suffices in most cases to study proper and faithful actions.
For later reference we record the following consequence.
\begin{Lem}\label{ActionOnDiscreteSet}
Let $X$ be a discrete topological space and let $G$ be a topological group that acts on $X$.
\begin{enumerate}[\rm(i),nosep]
 \item The action is continuous if and only if all point stabilizers are open.
 \item The action is proper if and only if all point stabilizers are compact and open.
 \item If $G$ is a discrete group, then the action is proper if and only if every point stabilizer is finite.
\end{enumerate}
\end{Lem}
\begin{Lem}\label{ProperActionEquivariant}
Let $G\times X\longrightarrow X$ and $K\times Y\longrightarrow Y$ be continuous actions of topological
groups, let $\alpha:G\longrightarrow K$ be a morphism of topological groups, and
let $f:X\longrightarrow Y$ be an equivariant and proper map, so that the diagram
\[
\begin{tikzcd}
 G\times X \arrow{r}\arrow{d}[swap]{\alpha\times f} & X \arrow{d}{f} \\
 K\times Y \arrow{r} & Y.
\end{tikzcd}
\]
commutes.
If the action of $G$ is proper, then $\alpha$ is a proper map.
In particular, $\alpha$ is closed and has a compact kernel.

Conversely, if $\alpha$ is closed with compact kernel, and if the  $K$-action is
proper, then the $G$-action is also proper.
\end{Lem}
\begin{proof}
 We consider the commutative diagram
 \[
\begin{tikzcd}
 G\times X \arrow{r}\arrow{d}[swap]{\alpha\times f} & X\times X \arrow{d}{f\times f} \\
 K\times Y \arrow{r} & Y\times Y.
\end{tikzcd}
\]
Assume that the $G$-action is proper.
By \ref{ProperFacts}, the map $f\times f$ is proper and therefore $\alpha\times f$ is also proper.
For $x\in X$, the restriction-corestriction $G\times\{x\}\longrightarrow K\times\{f(x)\}$
is also proper. Hence $\alpha$ is closed, with compact kernel.

If $\alpha$ is closed and with compact kernel, then $\alpha$ is proper and therefore
$\alpha\times f$ is also proper. Then $G\times X\longrightarrow X\times X$ is proper by~\ref{ProperFacts},
provided that $K\times Y\longrightarrow Y\times Y$ is proper.
\end{proof}
\begin{Cor}\label{InducedIsProper}
 Suppose that $K\times Y\longrightarrow Y$ is a proper action. If $G\subseteq K$ is a
 closed subgroup
 and if $X\subseteq Y$ is a closed  $G$-invariant subset, then the restricted action 
 $G\times X\longrightarrow X$ is also proper.
\end{Cor}

The following basic example shows that a continuous action with trivial stabilizers
need not be proper. 
\begin{Ex}\label{RonT}
Let $S=\{u\in \CC\mid \mathbb |u|=1\}$ and let $a$ be an irrational real number.
Then $\mathbb R$ acts freely and continuously
on $S\times S$ via $(t,u,v)\longmapsto (\exp(\sqrt{-1}t)u,\exp(\sqrt{-1}at)v)$, and the same holds
for the discrete group $\ZZ\subseteq\RR$. The orbits of these two actions are dense
and therefore these actions are not proper.
\footnote{This follows also from the fact that $B=C=S$ is compact, while 
$\RR=\RR_{S,S}$ is not.}
\end{Ex}
A topological group acting properly on a locally compact space $X$ is necessarily locally compact
by~\ref{ProperFacts}. Also, every compact (eg. finite) group that acts continuously acts properly.
However, proper actions are not necessarily related to local compactness.

\begin{Ex}\label{Banach}
Let $G$ be a metrizable topological group. Then the left regular action of $G$ on itself is proper.
Indeed, if $A,B\subseteq G$ are compact, then $G_{A,B}=\{ba^{-1}\mid a\in A, b\in B\}=BA^{-1}$ is compact.
\end{Ex}

\section{Isometry groups of proper metric spaces}

Let $(X,d)$ be a metric space. The isometric embeddings $X\longrightarrow X$
form a monoid, which we denote by $\IE(X)$. The group of invertible elements
in this monoid is the isometry group $\Iso(X)\subseteq\IE(X)$. We note that
every $g\in\IE(X)$ is a proper map,

The following is well-known
\cite[X.\S3.5]{BourTop}. For the sake of completeness, we include the proof.
For a subset $Y\subseteq X$ and $\eps>0$ we put \[B_\eps(Y)=\{x\in X\mid d(x,y)<\eps\text{ for some }y\in Y\}.\]
\begin{Lem}\label{IsoLemma}
Let $X$ be a metric space. Then
the compact-open topology and the topology of pointwise convergence coincide on $\IE(X)$.
With respect to this topology,
$\IE(X)$ is a topological monoid, the action $\IE(X)\times X\longrightarrow X$
is continuous and $\Iso(X)$ is a topological group. 
\end{Lem}
\begin{proof}
For $A,U\subseteq X$ we put $\langle A;U\rangle=\{g\in\IE(X)\mid g(A)\subseteq U\}$.
Then the sets $\langle F;U\rangle$, where $F$ is finite and $U$ is open form a
subbasis for the topology of pointwise convergence. The sets $\langle K;U\rangle$,
where $K$ is compact and $U$ is open form a subbasis for the compact-open topology.
Since every finite set is compact, this shows that the topology of pointwise
convergence is coarser than the compact-open topology.

Let $W\subseteq \IE(X)$ be an open set in the compact-open topology, 
with $g\in W$. Then there are compact sets $K_1,\ldots,K_n\subseteq X$ and 
open sets $U_1,\ldots,U_n\subseteq X$, with 
\[
g\in\langle K_1;U_1\rangle\cap\cdots\cap\langle K_n;U_n\rangle\subseteq W.
\]
Since the $K_i$ are compact, there exists $\eps>0$ such that
$g(B_\eps(K_i))\subseteq U_i$ holds for $i=1,\ldots,n$. 
Moreover, there are finite subsets 
$F_i\subseteq K_i$ such that 
$K_i\subseteq B_{\eps/3}(F_i)$. 
If $h\in\IE(X)$ with $d(gp,hp)<\eps/3$ for all 
$p\in F_1\cup\cdots\cup F_n=F$,
then $d(gq,hq)<\eps$ holds for all $q\in K_1\cup\cdots\cup K_n$ and thus $h(K_i)\subseteq U_i$
holds for $i=1,\ldots,n$,
that is, 
\[
 g\in\bigcap_{p\in F}\langle\{p\};B_{\eps/3}(f(p))\rangle
 \subseteq 
 \langle K_1;U_1\rangle\cap\cdots\cap\langle K_n;U_n\rangle\subseteq W.
\]
This shows that the compact-open topology is coarser than the topology of pointwise
convergence. Hence both topologies agree. For the remainder of the proof, we work with
the topology of pointwise convergence.

Let $p\in X$ and $g,h,g',h'\in\IE(X)$. If $d(gp,g'p)<\eps/2$ and if $d(hgp,h'gp)<\eps/2$,
then 
\[
 d(hgp,h'g'p)\leq 
 d(hgp,h'gp)+d(h'gp,h'g'p)<
 \eps.
\]
Therefore the composition $\IE(X)\times\IE(X)\longrightarrow\IE(X)$ is continuous.

Let $(g,p)\in\IE(X)\times X$. If $h\in\IE(X)$ and $q\in X$, with $d(p,q)<\eps/2$
and $d(gp,hp)<\eps/2$, then $d(gp,hq))\leq d(gp,hp)+d(hp,hq)<\eps$, hence the joint evaluation
map $\IE(X)\times X\longrightarrow X$ is continuous.

Suppose that $g\in\Iso(X)$ and that $p\in X$. Put $q=g^{-1}(p)$.
If $h\in\Iso(X)$ with $d(hq,gq)=d(hq,p)<\eps$, then $d(h^{-1}p,g^{-1}p)=d(p,hq)<\eps$.
We have shown that $h^{-1}\in\langle \{p\};B_\eps(g^{-1}p)\rangle$, provided that
$h\in\Iso(X)\cap\langle \{q\};B_\eps(gq)\rangle$. Hence inversion is continuous in $\Iso(X)$.
\end{proof}
From now on we endow the monoid $\IE(X)$ with the topology of pointwise convergence.
We recall that a metric space $(X,d)$ is called \emph{proper} if the Heine--Borel Theorem holds
in $X$: a subset is compact if and only if it is bounded and closed.
Proper metric spaces are locally compact and complete.
It follows from Lemma~\ref{ProperLemmma} that a metric space $(X,d)$ is proper if and only if 
the map $(x,y)\longmapsto (d(x,y),y)$ from $X\times X$ to $\RR\times X$ is proper.

Parts of following result are proven in \cite[Thm.~3.1]{AMN}, in \cite[5.2--5.6]{Gao} 
for separable spaces, and in 
\cite[Thm.~2.3]{GKVW}~\footnote{The proof given in \cite{GKVW} contains a gap; we fail to explain why $\Iso(X)\subseteq\IE(X)$
is closed.}.
\begin{Thm}\label{IsoPropThm}
Let $X$ be a proper metric space. Then the map $\IE(X)\times X\longrightarrow X\times X$
that maps $(g,x)$ to $(gx,x)$ is proper.
The topological monoid $\IE(X)$ is locally compact and second countable.
The subgroup $\Iso(X)$ is closed in $\IE(X)$ and hence also locally compact and second countable. 
Moreover, the action
$\Iso(X)\times X\longrightarrow X$ is proper.
\end{Thm}

\begin{proof}
We first show that the map $F:(g,x)\longmapsto (d(gx,x),x)$ is proper. 
If $C\subseteq\RR$ and $B\subseteq X$
are compact, we claim that the set 
\[\IE(X)_{B,C}=\{g\in\IE(X)\mid \text{ there is }b\in B\text{ such that }d(gb,b)\in C\}\]
has compact closure.
We choose $r>0$ with $C\subseteq [-r,r]$, and with $\mathrm{diam}(B)\leq r$.
For $g\in \IE(X)_{B,C}$ and $b\in B$ with $d(gb,b)\in C$
and $z\in X$ we have  then 
\[d(gz,z)\leq d(gz,gb)+d(gb,b)+d(b,z)\leq r+2d(z,b)\leq 3r+2d(z,B).\]
Hence $\IE(X)_{B,C}\subseteq\prod_{z\in X}\bar B_{3r+2d(z,B)}(z)\subseteq \prod_{z\in X}X$ has compact closure.
It follows from Lemma~\ref{ProperLemmma} that $F$ is proper.
Since $F$ factors as $(g,x)\longmapsto (gx,x)\longmapsto (d(gx,x),x)$, the map
$(g,x)\longmapsto (gx,x)$ is also proper by \ref{ProperFacts}.
In particular, $\IE(X)\times X$ is locally compact and thus 
$\IE(X)$ is locally compact.

Every compact metrizable space is second countable, cp.~\cite[XI.4.1]{Dug}. 
Being proper, $X$ is therefore a union
of countably many open and second countable sets. Hence $X$ is second countable,
and therefore the compact-open topology on $\IE(X)$ is also second countable
\cite[XII.5.2]{Dug}.

Suppose that $g\in \IE(X)$ is contained in the closure of $\Iso(X)$. We fix $p\in X$ and $r>0$, such
that $d(p,gp)<r$.
The set $K=\{h\in\IE(X)\mid d(p,hp)\leq r\}$ is compact and contains $g$.
We put $L=K\cap\Iso(X)$ and we note that $g\in\overline{L}$. 
We claim that $\id_X\in g\overline{L}$.
For every neighborhood $V$ of $g$, there is an element $h\in L\cap V$.
But then also $h^{-1}\in L$ and thus $\id_X\in VL\subseteq V\overline{L}$.
Since $\overline{L}$ is compact, Wallace' Lemma \cite[XI.2.6]{Dug}, \cite[3.2.10]{Eng} shows that 
$\id_X\in g\overline{L}$.
In particular, there exists $h\in\IE(X)$ such that $gh=\id_X$. Thus $g$ is surjective,
and hence in $\Iso(X)$. We have shown that $\Iso(X)\subseteq\IE(X)$ is closed.

Since $\Iso(X)\subseteq\IE(X)$ is closed, the restriction $\Iso(X)\times X\longrightarrow X\times X$
is proper by~\ref{ProperFacts}.
\end{proof}
The following example shows that the assumption of $X$ being proper is essential in the previous theorem.
\begin{Ex}
Let $X$ be a countably infinite set, with the discrete metric
$d(x,y)=1$ whenever $x\neq y$. Then $X$ is locally compact, but not proper.
The monoid $\IE(X)$ consists of all injective maps $X\longrightarrow X$,
while the isometry group of $X$ 
consists of all permutations of $X$. For every $g\in\IE(X)$ and every finite
subset $E\subseteq X$, there exists a permutation $h$ of $X$ that agrees with 
$g$ on $E$. Thus $\Iso(X)$ is dense in $\IE(X)$.
Neither $\IE(X)$ nor $\Iso(X)$ is locally compact.
\footnote{Both spaces are Polish. 
The monoid $\IE(X)$ is the left Weil completion of $\Iso(X)$,
while the group $\Iso(X)$ is Raikov complete.}
\end{Ex}
Example \ref{RonT} shows that a group can act isometrically and faithfully on a proper metric space
without acting properly. However, every group that acts faithfully, isometrically and properly
on a proper metric space carries necessarily the topology of pointwise convergence,
as we show now.
\begin{Prop}\label{TopologyofG}
Let $(X,d)$ be a proper metric space and let $G\times X\longrightarrow X$ be a proper
and isometric action of a topological group $G$.
Then the associated homomorphism $G\longrightarrow\Iso(X)$ is continuous and proper.
In particular, $G$ is locally compact.
If the kernel of the action is trivial, then $G$ carries the topology of pointwise
convergence.
\end{Prop}
\begin{proof}
For every $z\in X$, the evaluation map $G\longrightarrow X$ that maps 
$g$ to $gz$ is continuous. Hence the homomorphism $G\longrightarrow\Iso(X)$
is continuous with respect to the topology of pointwise convergence on $\Iso(X)$.
The first claim follows now from Lemma~\ref{ProperActionEquivariant}, with 
$K=\Iso(X)$ and $f=\id_X$.
The second claim follows from the fact that an injective closed map is a topological embedding.
\end{proof}
We recall that a group $G$ acts \emph{properly discontinuously} on a space $X$
if the action is proper with respect to the discrete topology on $G$.
\footnote{Again, different authors define properly discontinuous actions in rather different 
ways. Kapovich's article \cite{Kapo} is an excellent reference for a comparison of
differing definitions.
Some authors require in addition that the action is free.}
\begin{Cor}
Suppose that a group $G$ acts faithfully and isometrically on a
proper metric space $(X,d)$. Then the following are equivalent.
\begin{enumerate}[\rm(i),nosep]
 \item The action is properly discontinuous.
 \item The group $G$ is discrete in the topology of pointwise convergence.
 \item The group $G$ is discrete in the compact-open topology.
\end{enumerate}
\end{Cor}
We note that condition (ii) may be stated as follows: there are points $p_1,\ldots,p_n\in X$
and $\eps>0$ such that every $g\in G$ with $d(gp_j,p_j)<\eps$, for $j=1,\ldots,n$,
is necessarily the identity element.

\section{Metric simplicial complexes}

\begin{Def}
Let $V$ be a set. A \emph{simplicial complex} $\Delta$ with vertex set $V$
is a set of finite subsets of $V$ that satisfies the following axioms.
\begin{enumerate}[\rm(i),nosep]
\item If $a\subseteq b\in\Delta$, then $a\in\Delta$.
 \item $\bigcup\Delta =V$.
\end{enumerate}
The \emph{geometric realization} $|\Delta|$ of $\Delta$ is defined as follows, cp.~\cite[3.1.14]{Spanier}.
In the real vector space $\RR^{(V)}$ spanned by the basis $V$, we let $|a|$ denote the convex hull of 
the simplex $a\in\Delta$.
The geometric realization of $\Delta$ is the set $|\Delta|=\bigcup\{|a|\mid a\in\Delta\}$.
The subsets $|a|\subseteq|\Delta|$ will also be called simplices.
The \emph{weak topology} on $|\Delta|$ is defined as follows.
Each simplex $|a|$ carries its natural compact topology as a subset of $\RR^a$,
and a subset $A\subseteq |\Delta|$ is closed if $A\cap|a|$ is closed in $|a|$ for all $a\in\Delta$.
\end{Def}
\begin{Lem}\label{Vcloseddiscreteweak}
Let $\Delta$ be a simplicial complex. Then the vertex set $V\subseteq|\Delta|$
is closed and discrete in the weak topology.
\end{Lem}
\begin{proof}
The vertex set is closed from the definition of the weak topology.
If $v$ is a vertex, then \[U=|\Delta|\setminus\bigcup\{|a|\mid a\in\Delta\text{ and }v\not\in a\}.
\]
is an open neighborhood of $v$ which contains no vertex besides $v$.
\end{proof}
Suppose that $\Delta'$ is another simplicial complex, with vertex set $V'$.
A map ${\phi:V\longrightarrow V'}$ is called \emph{simplicial} if it maps simplices to simplices.
Then $\phi$ extends to a map ${\phi:\Delta\longrightarrow\Delta'}$, and to a linear map $\phi:|\Delta|\longrightarrow|\Delta'|$
which is continuous with respect to the weak topologies.
\begin{Lem}\label{coeqpw}
Let $\Delta$ and $\Delta'$ be simplicial complexes on vertex sets $V,V'$. Then the following topologies
on the set $S(\Delta,\Delta')$ of all simplicial maps from $\Delta$ to $\Delta'$ coincide.
\begin{enumerate}[\rm(i),nosep]
 \item The compact-open topology, for the weak topologies on $|\Delta|$ and $|\Delta'|$.
 \item The topology of pointwise convergence, as maps from $|\Delta|$ to $|\Delta'|$, for any topology on $|\Delta'|$
 in which the vertex set $V'\subseteq|\Delta'|$ is discrete.
 \item The topology of pointwise convergence, as maps from $V$ to the discrete space $V'$.
\end{enumerate}
\end{Lem}
Case (ii) in the lemma applies in particular to the weak topology on $|\Delta'|$
by Lemma~\ref{Vcloseddiscreteweak}.
\begin{proof}
We first show that the topologies in (ii) and (iii) on $S(\Delta,\Delta')$ are equal.
If $\phi,\psi\in S(\Delta,\Delta')$ agree on a finite subset $V_0\subseteq V$,
then they agree on the finite subcomplex $\Delta_0$ spanned by the vertex set $V_0$.
Hence the topology in (ii) is coarser than the topology in (iii).
On the other hand, if $\phi\in S(\Delta,\Delta')$ and if $V_0\subseteq V$ is finite,
then we can choose (in the topology given on $|\Delta'|$) 
for every $v\in V_0$ a neighborhood $U_v$ of $\phi(v)$ that contains
no vertex besides $\phi(v)$. Hence if $\psi\in S(\Delta,\Delta')$ and if $\psi(v)\in U_v$ holds for all $v\in V_0$,
then $\phi$ and $\psi$ agree on $V_0$. This shows that the topology in (iii) is coarser than
the topology in (ii).

Now we compare the topologies in (i) and (iii).
Suppose that $A\subseteq |\Delta|$ is compact in the weak topology, 
that $U\subseteq|\Delta'|$ is open in the weak topology,
and that $\phi\in S(\Delta,\Delta')$ is a simplicial map with
$\phi(A)\subseteq U$. Since $A$ is compact, there is  a finite subcomplex
$\Delta_0\subseteq\Delta$ with $A\subseteq|\Delta_0|$, cp.~\cite[Ch.~3.1.19]{Spanier}.
Let $V_0$ denote the (finite) vertex set of $\Delta_0$.
If $\psi$ is a simplicial map which agrees with $\phi$ on $V_0$, then 
$\psi|_{\Delta_0}=\phi|_{\Delta_0}$. In particular,
$\psi(A)=\phi(A)\subseteq U$. 
This shows that the compact-open topology in (i) is coarser than the topology in (iii).

The topology in (iii) agrees with the topology of pointwise convergence in (ii), for the
weak topology on $|\Delta'|$. Since the topology of pointwise convergence is always
coarser than the compact-open topology, all three topologies are equal.
\end{proof}
Later we will be interested in the joint continuity of the evaluation map 
\[
 S(\Delta,\Delta')\times|\Delta|\longrightarrow|\Delta'|.
\]
This requires some preparations.
\begin{Def}\label{LocallyFiniteTopology}
We call a topology on the geometric realization $|\Delta|$ of a simplicial complex $\Delta$ \emph{locally finite}
if every $x\in|\Delta|$ has some neighborhood $N_x$ which is contained in the geometric
realization of a finite subcomplex $\Delta_x\subseteq\Delta$, i.e.
$N_x\subseteq|\Delta_x|$. 
In this case, every compact subset $C\subseteq|\Delta|$ is contained in the geometric realization of
a finite subcomplex of $\Delta$.
\end{Def}
For example, the discrete topology on $|\Delta|$ is locally
finite. If the simplicial complex $\Delta$ is locally finite
(every vertex is contained in only finitely many simplices), then the weak topology
on $|\Delta|$ is locally finite.
\begin{Prop}\label{Eval}
Let $\Delta$ and $\Delta'$ be simplicial complexes with topologies on
$|\Delta|$ and $|\Delta'|$, respectively.
Let $M\subseteq S(\Delta,\Delta')$ be a set of simplicial maps which are continuous with respect to these
topologies. Then we have the following.
\begin{enumerate}[\rm(i),nosep] 
 \item 
If the topology on $|\Delta|$ is locally finite, then the joint evaluation map 
\[
 M\times|\Delta|\longrightarrow|\Delta'|
\]
is continuous, where $M$ carries the topology of pointwise convergence on vertices.
\item
If the topology on $|\Delta|$ is locally finite and if $V'\subseteq |\Delta'|$ is discrete,
then the topology of pointwise convergence on vertices, the topology of pointwise convergence on $|\Delta'|$
and the compact-open topology as maps from $|\Delta|$ to $|\Delta'|$ coincide on $M$.
\end{enumerate}
\end{Prop}
\begin{proof}
First we show claim (i).
Let $(\phi,x)\in M\times|\Delta|$ and put $z=\phi(x)$. Let $W$ be any neighborhood of $z$.
From the continuity of $\phi$ and from our assumptions, there is a neighborhood $U$ of $x$ with $\phi(U)\subseteq W$, such that
$U$ is contained in the geometric realization of a finite subcomplex $\Delta_x$ of $\Delta$.
Let $V_x$ denote the vertex set of $\Delta_x$. Then the set $M_x$ consisting of all maps in $M$
which agree with $\phi$ on $V_x$ is an open neighborhood of $f$ in $M$.
For $(\psi,y)\in M_x\times U$ we have thus $\psi(y)=\phi(y)\in W$.

For claim (ii) we note that by Lemma~\ref{coeqpw} the topology of pointwise convergence on vertices
coincides with the topology of pointwise convergence on $|\Delta'|$. 
The topology of pointwise convergence is
coarser than the compact-open topology. Suppose that $A\subseteq|\Delta|$ is compact. Then $A$ is contained in
the geometric realization of a finite subcomplex $\Delta_0$ of $\Delta$ and we may argue as in the proof of Lemma~\ref{coeqpw}. 
If $\phi\in M$ maps $A$ into the open set 
$W\subseteq|\Delta'|$, then every $\psi\in M$ which agrees with $\phi$ on the finite vertex set $V_0$ of $\Delta_0$
maps $A$ into $W$. Hence the compact-open topology is coarser than the topology of pointwise convergence.
\end{proof}
\begin{Prop}\label{ProperOnVertices}
Suppose that a topological group $G$ acts on a simplicial complex $\Delta$, with vertex set $V$. Suppose also that 
$|\Delta|$ carries a topology which is $G$-invariant, locally finite, such that the vertex set $V\subseteq|\Delta|$ closed and discrete,
and such that $|\Delta|\times|\Delta|$ is a $k$-space.
Then the following are equivalent.
\begin{enumerate}[\rm(i),nosep]
 \item The action of $G$ on $V$ is continuous and proper.
 \item The action of $G$ on $|\Delta|$ is continuous and proper.
\end{enumerate}
If these conditions hold, then $G$ is locally compact.
\end{Prop}
\begin{proof}
If (ii) holds, then (i) follows at once from Corollary~\ref{InducedIsProper}, and $G$ is locally compact because the discrete set $V$ is locally compact.

Suppose that the action of $G$ on $V$ is continuous and proper. Let $N$ denote the compact kernel of this action, and put $H=G/N$.
Then $H$ acts properly on $V$, and it suffices to show that the action of $H$ on $|\Delta|$ is proper.
The action of $H$ on $|\Delta|$ is continuous by Proposition~\ref{Eval}(i), since
$H$ carries by Proposition~\ref{TopologyofG} the topology of pointwise convergence on vertices.
Let $A,B\subseteq|\Delta|$ be compact sets. In view of Proposition~\ref{ProperActionProperties}
we have to show that the set $G_{A,B}\subseteq G$
is compact. There are finite subcomplexes $\Delta_0,\Delta_1\subseteq\Delta$, with vertex sets $V_0,V_1$, such that 
$A\times B\subseteq|\Delta_0|\times|\Delta_1|$.
Then $G_{A,B}\subseteq\bigcup\{G_{\{x\},\{y\}}\mid (x,y)\in V_0\times V_1\}$ has compact closure.
\end{proof}
The following example shows that one needs in general assumptions like local finiteness
on the topology $|\Delta|$.
\begin{Ex}
We consider the group $G=\QQ/\ZZ$ with the discrete topology. It acts in a natural way freely on the unit circle $V=\bS^1$. 
We may view $V$ as a $0$-dimensional simplicial complex $\Delta$, with geometric realization 
$|\Delta|=V$.
The discrete topology on $|\Delta|$ (which coincides with the weak topology)
is locally finite and the action of $G$ is proper, cp.~Lemma~\ref{ActionOnDiscreteSet}.
But if we put the usual compact topology of $\bS^1$ on $|\Delta|$, then the action of $G$ fails to be proper,
since $G=G_{V,V}$ is not compact. Indeed, the compact topology on $|\Delta|$
is not locally finite.
\end{Ex}

Now we turn to metrics on simplicial complexes.
The following is a basic example of a piecewise linear locally euclidean metrizable simplicial complex
which fails to be proper.
\begin{Ex}\label{NotProper}
Put
\[V=\{{\textstyle\frac{1}{n}}\mid\text{ for }n=1,2,3,4,\ldots\}\subseteq (0,1]\]
and let $\Delta$ be the $1$-dimensional simplicial complex whose $1$-simplices are the sets 
$\{\frac{1}{n+1},\frac{1}{n}\}$, for $n=1,2,3,4,\ldots$.
Then $|\Delta|$ is in the weak topology homeomorphic to the half-open interval $(0,1]$.
For the standard euclidean metric, $|\Delta|$ is not complete and therefore not a proper
metric space, even though $\Delta$ is a locally finite simplicial complex, and the topology on $\Delta$ is locally finite.
\end{Ex}
We recall the definition of an \emph{$M_\kappa$-simplicial complex} \cite[I.7]{BH}.
\begin{Def}
Suppose we are given a simplicial complex $\Delta$ and a real number $\kappa$. 
A \emph{geodesic simplex of constant sectional curvature $\kappa$} is the convex hull of a finite set of points
in general position in the simply connected Riemannian manifold $M_\kappa$ of constant sectional
curvature $\kappa$ \cite[I.7.1]{BH}.
An \emph{$M_\kappa$-structure} on $\Delta$ consists of a collection $\cS$ of geodesic simplices of constant sectional curvature $\kappa$.
For each $a\in\Delta$, there is a geodesic simplex $s_a\in\cS$
and an affine bijection \[\sigma_a:|a|\longrightarrow s_a\] in the sense of \cite[I.7A.7]{BH}.
These bijections are subject to the usual compatibility condition:
if $b\subseteq a$, then $\sigma_a\circ \sigma_b^{-1}:s_b\longrightarrow s_a$ maps 
$s_b$ isometrically onto a face of $s_a$ \cite[I.7A.9]{BH}.
In this way, every simplex $|a|$ carries a well-defined metric $d_{a}$ of constant sectional curvature $\kappa$.
The associated \emph{intrinsic pseudometric} on $|\Delta|$ is defined as follows.
An \emph{$m$-string} $x=(x_0,\ldots,x_m)$ is a finite sequence in $|\Delta|$, such that consecutive points
$x_{i-1},x_{i}$ are contained in a common simplex $|a_i|$.
The \emph{length} of the string is \[\ell(x)=\sum_{i=1}^m d_{a_i}(x_{i-1},x_i).\]
The \emph{intrinsic pseudo-metric distance} $d(x,y)$ between $x,y\in|\Delta|$
is then the infimum of the lengths of all strings joining $x$ and $y$.
If $\cS$ is finite, one says that the $M_\kappa$-structure
on $\Delta$ has \emph{finitely many shapes}. In this case
$(|\Delta|,d)$ is a complete geodesic metric space \cite[I.7.19]{BH}.
\end{Def}
The simplicial complex in Example~\ref{NotProper} is locally finite, but not with finitely many shapes.
\begin{Lem}\label{Vcloseddiscreted}
Let $\Delta$ be a simplicial complex with an $M_\kappa$-structure with finitely many shapes.
Then the intrinsic distance between distinct vertices is bounded away from $0$,
and hence $V$ is closed and discrete in the metric space $|\Delta|$.
\end{Lem}
\begin{proof}
For a vertex $v$ we put as in \cite[Def.~I.7.8]{BH} 
\[\eps(v)=\inf\{\eps(v,a)\mid v\in a\},\]
where
\[\eps(v,a)=\inf\{d_a(v,|b|)\mid b\subseteq a \text{ and }v\not\in b\}.\]
Since $\cS$ is finite, $\eps(v)$ is bounded away from $0$ as $v$ varies in $V$.
Now \cite[Lem.~I.7.9]{BH} says that if $w$ is another vertex with 
$d(v,w)<\eps(v)$, then $v=w$.
\end{proof}
We recall that a simplicial complex is locally finite if no vertex is contained in
infinitely many simplices. The following result appears to be folklore, but we could not find
a reference.
\begin{Thm}\label{MkappaProper}
Let $\Delta$ be a simplicial complex with an $M_\kappa$-structure with
finitely many shapes. Then the following are equivalent.
\begin{enumerate}[\rm(i),nosep]
 \item $\Delta$ is locally finite.
 \item $|\Delta|$ is a proper metric space with respect to the intrinsic metric. 
\end{enumerate}
If $\Delta$ is locally finite, then the weak topology and the metric topology on $|\Delta|$ coincide.
In particular, the topology is locally finite in the sense of Definition~\ref{LocallyFiniteTopology}.
\end{Thm}
\begin{proof}
For each simplex $a\in\Delta$,
the map $\sigma_a:|a|\longrightarrow s_a$ is a homeomorphism, where
$|a|$ carries the metric topology coming from $d$ \cite[I.7.6]{BH}.
In particular, each simplex $|a|$ is compact.
Since $\cS$ is finite, the distance between any two points $x,y\in|\Delta|$
is realized as the infimum of the lengths of all \emph{taut}
strings 
\footnote{An $m$-string $(x_0,\ldots,x_m)$ in $|\Delta|$ is taut if no three consecutive points $x_{i-1},x_i,x_{i+1}$ are in a common simplex,
and if all triples $x_{i-1},x_i,x_{i+1}$ are 'straight' in a weak sense.}
between $x$ and $y$, cp.~\cite[I.7.20 and I.7.24]{BH}.

Suppose that $\Delta$ is locally finite, that $x\in|\Delta|$, and that $r>0$. We claim that the 
closed ball $\bar B_r(x)$ is compact.
For every $r>0$ there is a number $N>0$ such that every taut 
$m$-string $(x_0,x_1,\ldots,x_m)$ between two points $x=x_0,y=x_m$ at distance $d(x,y)\leq r$
has $m\leq N$ \cite[I.7.28]{BH}.
It follows that there is a sequence of simplices
$a_1,\ldots,a_m\in\Delta$, with $x_{i-1},x_i\in |a_i|$.
Since $\Delta$ is locally finite, the set of all such $m$-chains
of simplices starting at $x$ with $m\leq N$ is finite.
Hence $\bar B_r(x)$ is contained in a finite subcomplex $\Delta_0$,
that is, $\bar B_r(x)\subseteq|\Delta_0|$. 
Since all simplices $|a|$ are compact, $|\Delta_0|$ is compact in the metric topology.

Suppose that $\Delta$ is not locally finite. The vertex set $V$ is closed and discrete by Lemma~\ref{Vcloseddiscreted}.
By assumption, there is a vertex $v$ that is contained in infinitely many simplices.
Put \[r=\sup\{d(v,w)\mid v,w\in V\text{ and } v,w\text{ are in a common simplex}\}.\] 
Since $\cS$ is finite, $r<\infty$.
Then the closed ball $\bar B_r(v)$ contains
the infinite closed discrete set $V\cap \bar B_r(v)$.
In particular, $\bar B_r(v)$ is not compact.

For the last claim we note that the metric topology on $|\Delta|$ is coarser
than the weak topology. Let $A\subseteq|\Delta|$ be compact in the metric topology.
Then $A$ is contained in some closed ball $\bar B_r(v)$.
The argument above shows that $\bar B_r(v)$ is contained in a finite subcomplex
$|\Delta_0|$. In particular, $A$ is contained in $|\Delta_0|$,
and $|\Delta_0|$ is compact in the weak topology.
Since $A$ is also closed in the weak topology, $A$ is compact in the weak topology.
Therefore the identity is a proper map from the weak topology to the metric topology,
and hence a homeomorphism.
\end{proof}
The next result is also well-known and widely used, but we did not find a reference.
\begin{Prop}\label{SimplcialIsos}
Let $\Delta$ be a simplicial complex with an $M_\kappa$-structure
with finitely many shapes, and let $d$ denote the intrinsic metric.
Then the subgroup $\Iso(|\Delta|)\cap S(\Delta,\Delta)$ consisting of all
simplicial isometries is a closed and totally disconnected subgroup of $\Iso(|\Delta|)$.

If $\Delta$ is locally finite, then $\Iso(|\Delta|)\cap S(\Delta,\Delta)$
is locally compact, totally disconnected, and second countable.
\end{Prop}
\begin{proof}
Let $g$ be an isometry which is in the closure of $\Iso(|\Delta|)\cap S(\Delta,\Delta)$,
and let $a$ be a simplex.
For each vertex $v\in a$ and every neighborhood $U$ of $g(v)$,
there is some $h\in \Iso(|\Delta|)\cap S(\Delta,\Delta)$ with $h(v)\in U$.
Since $V$ is closed and discrete by Lemma~\ref{Vcloseddiscreted},
the point $g(v)$ is a vertex. Hence there exists $h\in \Iso(|\Delta|)\cap S(\Delta,\Delta)$
with $g(v)=h(v)$ for all $v\in a$. It follows that $g|_{|a|}=h|_{|a|}$ and therefore 
$g$ is simplicial.

Since the space $V^V$ is totally disconnected and since $S(\Delta,\Delta)$ carries
the subspace topology by Lemma~\ref{coeqpw}, $\Iso(|\Delta|)\cap S(\Delta,\Delta)$ is totally disconnected.
The last claim follows from Theorem~\ref{MkappaProper} and Theorem~\ref{IsoPropThm}.
\end{proof}

\section{Buildings}
We recall the definition of a building as a chamber graph, as in 
Tits' \emph{Local approach} \cite{TitsLocal}.
The books \cite{AB,Br,Da,Ro,Weiss1,Weiss2} are excellent references.
Let $I$ be a finite set and let $(m_{i,j})_{i,j\in I}$
be a symmetric matrix with entries in $\{1,2,3,4,5,\ldots,\infty\}$, with
$m_{i,i}=1$ for all $i$ and $m_{i,j}\geq2$ for all $i\neq j$.
The associated \emph{Coxeter system} $(W,I)$ consists of the \emph{Coxeter group} 
$W$,
with the presentation
\[
 W=\langle I\mid (ij)^{m_{i,j}}=1\text{ if }m_{i,j}\neq\infty\rangle,
\]
and its generating set $I$. It follows that the product $ij$ has order $m_{i,j}$ in $W$, cp.~\cite[Lem.~2.1]{Ro}.
The \emph{length function} of $W$ with respect to the generating set $I$ is denoted by $\ell$.
A presentation of a group element $w=i_1\cdots i_m$ in terms of generators $i_1,\ldots,i_m\in I$ is called \emph{reduced}
or \emph{minimal} if $\ell(w)=m$.

\begin{Def}
Let $(W,I)$ be a Coxeter system.
A \emph{building} $\cB=(C,EC,t,\delta)$ of type $(W,I)$ consists of a  
simplicial graph $(C,EC)$ with vertex set $C$, edge set $EC$, and two maps
\[t:EC\longrightarrow I\text{ and }\delta:C\times C\longrightarrow W.\]
The elements of $C$ are called \emph{chambers}, the map $t$ is the \emph{edge coloring},
and the map $\delta$ is the \emph{$W$-valued distance function}. 
Two chambers which are adjacent by an edge of color $i$ are called \emph{$i$-adjacent}.
The graph $(C,EC)$ is called the \emph{chamber graph}.
A \emph{gallery} is a sequence of chambers
$(c_0,\ldots,c_m)$, such that $c_{s-1}$ is $i_s$-adjacent to $c_s$, for $s=1,\ldots,m$.
The \emph{type} of the gallery is the string $(i_1,\ldots,i_m)\in I^m$.
The maps $t$ and $\delta$ are subject to the following axioms.
\footnote{We view the empty graph as the unique building corresponding to $I=\emptyset$.}
\begin{enumerate}[label=(B{\arabic*}),nosep]
\item For every chamber $a\in C$ and $i\in I$, there is at least one chamber $b$ which
is $i$-adjacent to $a$. If another chamber $c$ is $i$-adjacent to $a$, then $b$ and $c$ are also $i$-adjacent.
 \item Suppose that $w\in W$ and that $w=i_1\cdots i_m$ is a reduced presentation.
 Then $\delta(a,b)=w$ holds for two chambers $a,b$ if and only if there is a gallery of type 
 $(i_1,\ldots,i_m)$ from $a$ to $b$.
\end{enumerate}
It follows that $\delta(a,b)=1$ holds if and only if $a=b$, and that 
$\delta(b,a)=\delta(a,b)^{-1}$.
\end{Def}
The cardinality of the set $I$ is called the \emph{rank} of the building.
We note that two chambers $a,b$ are $i$-adjacent if and only if
$\delta(a,b)=i$. Hence $\delta$ determines $EC$ and $t$ uniquely.
Conversely, $\delta$ is determined by $EC$ and $t$ by (B2).
Buildings can be thought of as generalizations
of Cayley graphs of Coxeter groups.
\begin{Ex}\label{ThinEx}
The most basic example of a building of type $(W,I)$ is the Cayley graph of a Coxeter system $(W,I)$.
Thus $C=W$ and $EC=\{\{w,wi\}\mid w\in W\text{ and }i\in I\}$, with $\delta(a,b)=a^{-1}b$ and 
$t(\{w,wi\})=\delta(w,wi)=i$.
This building has the particular property that for every chamber $a$ and every $i\in I$,
there is precisely one other chamber $b$ which is $i$-adjacent to $a$. 
Such buildings are called \emph{thin}.
One can show that every thin building of type $(W,I)$ is isomorphic to the Cayley graph of $(W,I)$.
(For the proof one fixes a chamber $c$ and considers the map $a\longmapsto\delta(c,a)$ from $C$ to $W$.)
A building is called \emph{thick} if for every chamber $a$ and every $i\in I$, there are at least
two distinct chambers $b,c$ which are $i$-adjacent to $a$. 
\end{Ex}
\begin{Ex}\label{RAB}
Let $(G_i)_{i\in I}$ be a finite family of nontrivial groups $G_i$, and let \[G=\coprod_{i\in I}G_i\]
denote their coproduct. If we put $C=G$ and if we call $a,b\in G$ $i$-adjacent if 
$1\neq a^{-1}b\in G_i\subseteq G$, then we obtain a building. The corresponding Coxeter system
has $m_{i,j}=\infty$ for all $i\neq j$. This building is thick if and only if every $G_i$ has at least $3$ elements.

The construction can be modified  \cite[18.1.10]{Da}. 
Suppose that $\Gamma$ is a simplicial graph on the vertex set $I$.
We put \[G_\Gamma=G/\langle\langle [G_i,G_j]\text{ if }i,j\text{ are adjacent in }\Gamma\rangle\rangle.\]
This group is called the \emph{graph product} of the $G_i$.
We put $C=G_\Gamma$ and we define $i$-adjacency as before by $1\neq a^{-1}b\in G_i\subseteq G_\Gamma$.
The associated Coxeter system has $m_{i,j}=\infty$ whenever $i$ and $j$ are different and not joined by an edge
in $\Gamma$,
and $m_{i,j}=2$ if $i$ and $j$ are joined by an edge.

Buildings of this type are called \emph{right-angled buildings}.
These buildings are very different from the so-called \emph{affine} or \emph{spherical buildings}.
The latter are, by Tits' fundamental classification results, closely related to semisimple algebraic groups  
\cite{Weiss1,Weiss2}. 
Right-angled buildings, on the other hand, are similar to trees.
They are easy to construct, while their automorphism groups have many interesting
subgroups.
\end{Ex}
A \emph{subbuilding} of a building is a full subgraph which is a building in its own right (of the same type $(W,I)$).
A subbuilding which is thin is called an \emph{apartment} in the ambient building. One can show that any two
chambers in a building are contained in some apartment \cite[Cor.~3.7]{Ro}. 
Every apartment is by Example~\ref{ThinEx} isomorphic to  the Cayley graph of $(W,I)$.
If $A,A'$ are two apartments, 
then there is a color-preserving graph isomorphism $A\longrightarrow A'$
that fixes $A\cap A'$ element-wise \cite[Thm.~3.11]{Ro}. 

Apartments can be viewed as coordinate charts in buildings,
similarly to the coordinate charts in a manifold.
An \emph{atlas} $\mathbf A$ is a set of apartments such that every pair of chambers of the building is
contained in at least one apartment in $\mathbf A$.
Every building has a unique maximal atlas, consisting of all apartments in the building.
\begin{Def}
An automorphism $h$ of a building is a graph automorphism that preserves the coloring of the edges.
Then the distance function $\delta$ is $h$-invariant, $\delta(ha,hb)=\delta(a,b)$.
Let $G$ be a group that acts on a building $\cB$ as a group of automorphisms.
The action is \emph{Weyl-transitive} if 
for all $w\in W$ and all chambers $a,b,a',b'$ with $w=\delta(a,b)=\delta(a',b')$ there is some $g\in G$ with $ga=a'$ and $gb=b'$.
Let $\mathbf A$ be an atlas. The action of of the group $G$ is \emph{strongly transitive} on the atlas $\mathbf A$
if the action preserves $\mathbf A$, and if for every $a\in A\in \mathbf A$
and $a'\in A'\in\mathbf A$ there is some $g\in G$ with $ga=a'$ and $gA=A'$.
Strongly transitive actions are in particular Weyl-transitive.
The converse is in general not true~\cite{AB1}.
\end{Def}
In Example~\ref{ThinEx}, the Coxeter group $W$ acts (from the left) on itself, and this 
action is strongly transitive and Weyl-transitive.
In Example~\ref{RAB}, the automorphism group of the right-angled building (which contains $G$ resp. $G_\Gamma$ as a subgroup)
acts strongly transitively on the maximal atlas, cp.~\cite[Cor.~1.2]{Cap}.

A gallery in a building from a chamber $a$ to a chamber $b$ can be viewed as a simplicial path in the graph.
One can show that $\ell(\delta(a,b))$ is the graph-theoretic distance between $a$ and $b$, i.e. a gallery
$(c_0,\ldots,c_m)$ of type $(i_1,\ldots,i_m)$ is minimal if and only if $w=i_1\cdots i_m$ is reduced.
Moreover, one can show that a minimal gallery is uniquely determined by its endpoints and its type \cite[3.1]{Ro}.
Thus, a group $G$ acts Weyl-transitively on a building if and only if for every reduced type $(i_1,\ldots,i_m)$
and every pair of galleries $(c_0,\ldots,c_m)$ and $(c_0',\ldots,c_m')$ of the same type $(i_1,\ldots, i_m)$
there is a $g\in G$ with \[(gc_0,\ldots,gc_m)=(c_0',\ldots,c_m').\]
Now we turn to infinite galleries.
\begin{Def}
We call an infinite string $(i_s)_{s\in\ZZ}$ in $I^\ZZ$ \emph{reduced} if every finite substring $(i_s,i_{s+1},\ldots,i_{s+r})$
is reduced, i.e. if $\ell(i_si_{s+1}\cdots i_{s+r})=r+1$ for all $s\in\ZZ$ and $r\geq 0$.
One can show that the Coxeter group $W$ is infinite if and only if every finite reduced string $(i_1,\ldots,i_m)$
is part of an infinite reduced string~\cite[Thm.~2.16]{Ro}.
An \emph{infinite gallery} of type $(i_s)_{s\in\ZZ}$ is a sequence of chambers $(c_s)_{s\in\ZZ}$,
where  $c_{s-1}$ is $i_s$-adjacent to $c_s$, for all $s\in\ZZ$.
If $W$ is infinite, then every minimal gallery is part of an infinite reduced gallery in the building.
Moreover, every infinite reduced gallery is contained in some apartment in the building \cite[Thm.~3.6]{Ro}.

We say that a group $G$ acts \emph{infinitely Weyl-transitively} on a building if for every infinite reduced
string $(i_s)_{s\in\ZZ}$ in $I^\ZZ$ and every pair of infinite galleries $(c_s)_{s\in\ZZ}$ and $(c_s')_{s\in\ZZ}$
of type $(i_s)_{s\in\ZZ}$
there is an element $g\in G$ with $gc_s=c_s'$, for all $s\in\ZZ$.
\end{Def}
Suppose that $W$ is infinite and that $\bf A$ is the maximal atlas of a building $\cB$ of type
$(W,I)$. For a group $G$ acting by automorphisms on $\cB$, 
we have the implications
\[
 \text{Strongly transitive on $\bf A$}\quad\Longrightarrow\quad
 \text{infinitely Weyl-transitive}
 \quad\Longrightarrow\quad
 \text{Weyl-transitive}.
\]
If $W$ is finite, then every Weyl-transitive action is strongly transitive on the maximal atlas
(and the maximal atlas is the unique atlas).

\section{Proper actions on buildings}

A building is called \emph{locally finite} if for every color $i\in I$ and every
chamber $a$, the set of all chambers which are $i$-adjacent to $a$ is finite.
In Example~\ref{RAB} this holds if and only if the groups $G_i$ are  finite.
\begin{Thm}\label{AutBuilding}
Let $\cB$ be a building and put $d(a,b)=\ell(\delta(a,b))$.
Then $(C,d)$ is a complete metric space.
With respect to the topology of pointwise convergence on chambers, the automorphism group $\Aut(\cB)$
is a closed and totally disconnected subgroup of $\Iso(C)$.
If $\cB$ is locally finite, then $\Aut(\cB)$ is locally compact and the action on 
$C$ is proper.
\end{Thm}
We note that the topology of pointwise convergence on $\Aut(\cB)$ depends only on the discrete space $C$,
not on the structure of the building (or the specific metric $d$).
Since $\Aut(\cB)$ is contained in the totally disconnected space $C^C$,
it is totally disconnected.
\begin{proof}
We have $d(a,b)=n$ if and only if there is a minimal gallery of length $n$ from $a$ to $b$ \cite[3.1]{Ro}.
Hence $d$ is a (complete) metric on $C$. By Lemma~\ref{IsoLemma}, the isometry group of $(C,d)$ is a topological
group in the topology of pointwise convergence. Suppose that $g\in\Iso(C)$ is not an
automorphism of $\cB$. Then there exists a pair of $i$-adjacent chambers $a,b$
so that $ga,gb$ are $j$-adjacent, for some $j\neq i$.
The set $\{h\in\Iso(C)\mid ha=ga\text{ and }hb=gb\}$ is a neighborhood of $g$ and disjoint
from $\Aut(\cB)$. Hence $\Aut(\cB)$ is a closed subgroup in $\Iso(C)$.

If $\cB$ is locally finite, then the graph $(C,EC)$ is
locally finite and therefore every metric ball in $C$ is finite and hence compact. 
The isometry group $\Iso(C)$ of $(C,d)$ is therefore 
locally compact and acts properly on $C$ by Theorem~\ref{IsoPropThm}.
Thus $\Aut(\cB)$ is also locally compact and acts properly on $C$
by Corollary~\ref{InducedIsProper}.
\end{proof}
\begin{Thm}\label{inftrans}
Let $\cB$ be a building of type $(W,I)$, with $W$ infinite, and let $G$ be a topological group that acts properly on the discrete metric space $C$ as a group of automorphisms on $\cB$.
Then the following are equivalent.
\begin{enumerate}[\rm(i),nosep]
 \item The action is infinitely Weyl-transitive.
 \item The action is Weyl-transitive.
\end{enumerate}
\end{Thm}
\begin{proof}
As we observed in the last section, (i) implies (ii).

Suppose that the action is Weyl-transitive, that $(i_s)_{s\in\ZZ}$ is reduced and that $(c_s)_{s\in\ZZ}$ and $(c_s')_{s\in \ZZ}$ are two
infinite galleries, both of type  $(i_s)_{s\in \ZZ}$. For each $s\geq 0$
put \[G_s=\{g\in G\mid g(c_s)=c_s'\text{ and }g(c_{-s})=c_{-s}'\}.\]
The $G_s$ are compact (because the action on $C$ is proper) and nonempty (because the action is Weyl-transitive).
Moreover, $G_s\supseteq G_{s+1}$. Hence $\bigcap\{G_s\mid s\geq 0\}\neq\emptyset$.
\end{proof}
So far, we have considered buildings as edge-colored graphs. Now we turn to
geometric realizations.
Let $\cB$ be a building. The chamber graph $(C,EC)$ is then a $1$-dimensional simplicial complex,
which we denote by $\Delta_C$.
\begin{Prop}\label{TopologyFromDeltaC}
Let $\cB$ be a building, with chamber set $C$. Then the following topologies on $\Aut(\cB)$ coincide.
\begin{enumerate}[\rm(i),nosep]
 \item The topology of pointwise convergence on the discrete set $C$.
 \item The topology of pointwise convergence, for any topology on $|\Delta_C|$ in which $C$ is a discrete subset.
 \item The compact-open topology, for the weak topology  on $|\Delta_C|$.
\end{enumerate}
Suppose that $\Aut(\cB)$ carries this topology and that $|\Delta_C|$ carries some topology.
Then the joint evaluation map
\[
 \Aut(\cB)\times|\Delta_C|\longrightarrow|\Delta_C|
\]
is continuous if one of the following holds.
\begin{enumerate}[\rm(a),nosep]
 \item The topology on $|\Delta_C|$ is invariant and locally finite.
 \item The topology on $|\Delta_C|$ comes from an invariant metric $d$ in which $C$ is discrete.
\end{enumerate}
\end{Prop}
\begin{proof}
The claim about the topologies on $\Aut(\cB)$ follows at once from Lemma~\ref{coeqpw}.
If $|\Delta_C|$ carries an invariant locally finite topology, then joint evaluation is continuous by Proposition~\ref{Eval}.
If $d$ is an invariant metric and if $C$ is discrete, then $\Aut(\cB)\subseteq\Iso(|\Delta_C|)$ carries
the topology of pointwise convergence by (ii) and hence joint evaluation is continuous by Lemma~\ref{IsoLemma}.
\end{proof}
Now we compare proper actions on $C$ with proper actions on $|\Delta_C|$.
\begin{Prop}
Suppose that a topological group $G$ acts as a group of automorphisms on a building $\cB$ and that $|\Delta_C|$
carries a locally finite, $G$-invariant topology in which the chamber set $C$ is closed and discrete.
Assume also that $|\Delta_C|\times|\Delta_C|$ is a $k$-space.
The following are equivalent.
\begin{enumerate}[\rm(i),nosep]
 \item The action of $G$ on $C$ is continuous and proper.
 \item The action of $G$ on $|\Delta_C|$ is continuous and proper.
\end{enumerate}
If these hold, then $G$ is locally compact.
\end{Prop}
\begin{proof}
This follows from Proposition~\ref{ProperOnVertices}.
\end{proof}
There are (at least) two other important simplicial complexes associated to buildings,
the \emph{Tits complex} $\Delta_T$ and the (subdivided) \emph{Davis complex} $\Delta_D$.
\begin{Def}[The Tits complex]
Let $(W,I)$ be a Coxeter system and let $J\subseteq I$ be
a subset. Let $W_J\subseteq W$ denote the subgroup generated by $J$. Then $(W_J,J)$
is again a Coxeter system, for the restricted Coxeter matrix $(m_{i,j})_{i,j\in J}$
\cite[Cor.~2.14]{Ro}.
Let $c$ be a chamber in a building of type $(W,I)$ and let $J\subseteq I$.
The \emph{$J$-residue} $\Res_J(c)$ of $c$ is the subgraph consisting of all chambers that can be reached
from $c$ using galleries whose edge colors are in $J$, and the edges between these chambers. 
Then $\Res_J(c)$ is again a building of type $(W_J,J)$ \cite[Thm.~3.5]{Ro}. 

The poset $\{\Res_J(c)\mid c\in C\text{ and }J\subseteq I\}$,
ordered by reversed inclusion, is poset-isomorphic to a simplicial complex $\Delta_T$
(where $\Delta=\Res_I(c)$ corresponds to the empty set). 
The vertex set of $\Delta_T$
is 
\[
V=\{\Res_{I\setminus\{i\}}(c)\mid c\in C\text{ and }i\in I\},
\]
The simplices are finite subsets 
\[a=\{\Res_{I\setminus\{i_0\}}(c),\Res_{I\setminus\{i_1\}}(c),\ldots,
\Res_{I\setminus\{i_k\}}(c)\},\]
with $\{i_0,\ldots,i_k\}\subseteq I$,
and the simplex $a$ corresponds to $\bigcap a=\Res_{I\setminus\{i_0,\ldots,i_k\}}(c)$
in the poset isomorphism
\cite[5.6]{AB}. In particular,
the maximal simplices in $\Delta_T$ correspond to the chambers in $C$.
We call the geometric
realization $|\Delta_T|$ of $\Delta_T$ the \emph{Tits realization}.
This is the simplicial complex considered in \cite{TitsSpherical}.
\end{Def}
The Solomon-Tits Theorem \cite[Thm.~4.127]{AB} asserts that \[|\Delta_T|\simeq \{*\}\] is contractible (in the weak topology)
if $W$ is infinite.
If $W$ is finite, then it contains a (unique) element $w_0$ of maximal length.
Then $|\Delta_T|$ has the homotopy type of a wedge of spheres,
\[|\Delta_T|\simeq\bigvee_Q\bS^{n-1},\] where $n=\#I$ and $Q=\{a\in C\mid \delta(a,c)=w_0\}$,
for some (any) $c\in C$.

Earlier in this section, we endowed the automorphism group of a building with the topology 
of pointwise convergence on the chamber set $C$. This topology works well with the Tits realization.
\begin{Prop}\label{TitsRealization}
Let $\cB$ be a building, with chamber set $C$. Let $V$ be the vertex set of the Tits complex.
The following topologies on $\Aut(\cB)$ coincide.
\begin{enumerate}[\rm(i),nosep]
 \item The topology of pointwise convergence on the discrete set $C$.
 \item The topology of pointwise convergence on the discrete set $V$.
 \item The topology of pointwise convergence, for any topology on $|\Delta_T|$ in which $V$ is a discrete subset.
 \item The compact-open topology, for the weak topology  on $|\Delta_T|$.
\end{enumerate}
Suppose that $\Aut(\cB)$ carries this topology and that $|\Delta_T|$ carries some topology.
Then the joint evaluation map
\[
 \Aut(\cB)\times|\Delta_T|\longrightarrow|\Delta_T|
\]
is continuous if one of the following holds.
\begin{enumerate}[\rm(a),nosep]
 \item The topology on $|\Delta_T|$ is invariant and locally finite.
 \item The topology on $|\Delta_T|$ comes from an invariant metric $d$ in which $V$ is discrete.
\end{enumerate}
\end{Prop}
\begin{proof}
If two automorphisms agree on a finite set of chambers, then they agree on the finite
set of vertices of the corresponding maximal simplices, and vice versa.
Hence the topologies in (i) and (ii) coincide.
The rest of the proof is as in Proposition~\ref{TopologyFromDeltaC}.
\end{proof}
Proposition~\ref{ProperOnVertices} readily implies the following.
\begin{Prop}\label{ProperOnDelta_T}
Suppose that a topological group $G$ acts as a group of automorphisms on a building $\cB$ and that $|\Delta_T|$
carries a locally finite, $G$-invariant topology in which the vertex set $V$ is closed and discrete.
Assume also that $|\Delta_T|\times|\Delta_T|$ is a $k$-space.
The following are equivalent.
\begin{enumerate}[\rm(i),nosep]
 \item The action of $G$ on $V$ is continuous and proper.
 \item The action of $G$ on $|\Delta_T|$ is continuous and proper.
\end{enumerate}
If these hold, then $G$ is locally compact.
\end{Prop}
The following example shows that we cannot expect such a result for proper actions on chambers.
The reason is that the Tits complex of a locally finite building need not be locally finite.
\begin{Ex}
Let $X,Y,Z$ be nontrivial finite groups and consider the locally finite right-angled building $\Delta$ corresponding to the graph
product
\[G=X*Y*Z/\langle\langle[X,Z]\rangle\rangle,\]
as in Example~\ref{RAB}. The chambers are the elements of $G$ and the Tits realization is $2$-dimensional.
There are three types of vertices in $\Delta_T$,
corresponding to the cosets $V_1=G/(X*Y)$, $V_2=G/(Y*Z)$, and $V_3=G/(X\times Z)$.
The action of $G$ on the chamber set is free and thus the action of the discrete group $G$ on the chamber set $C$ is proper.
On the other hand, the stabilizer of a vertex in $V_1$ is isomorphic to $X*Y$ and hence infinite.
Therefore the action of $G$ on $|\Delta_T|$ in the weak topology is not proper.
\end{Ex}
The Tits complex $\Delta_T$ is mainly interesting if $W$ is finite (spherical) or of irreducible
euclidean type, because then $|\Delta_T|$ carries an interesting invariant CAT(1) resp. CAT(0) metric.
For many questions arising in geometric group theory, the Tits realization is not the right 
geometric object.
The Davis realization of a building $\cB$ is defined as follows.
\begin{Def}[The subdivided Davis complex]
A residue $\Res_J(c)$ in a building $\cB$  is called \emph{spherical} if $W_J$ is finite.
Let $V$ denote the set of all spherical residues of $\cB$. Let $\Delta_D$
denote the simplicial complex whose simplices are ascending
chains $\Res_{J_0}(c)\subseteq \Res_{J_1}(c)\subseteq\cdots\subseteq \Res_{J_m}(c)$,
with $J_0\subseteq J_1\subseteq\cdots\subseteq J_m$ and with $W_{J_m}$ finite.
A locally finite spherical building is finite. Therefore the subdivided Davis complex $\Delta_D$ is locally finite if and only if $\cB$ is locally finite.
The \emph{Davis realization} of $\cB$ is defined to be the set $|\Delta_D|$.
\footnote{The simplicial complex $\Delta_D$ is not quite the \emph{Davis complex}; rather, it is
the barycentric subdivision of the Davis complex. However, $|\Delta_D|$ is the geometric realization of the Davis complex.}
The Davis realization is contractible (with respect to the weak topology).
More importantly, it admits a CAT(0) metric.
This metric $d$ is the intrinsic metric of an $M_0$-structure on $\Delta_D$ with finitely many 
shapes~\cite[18.3]{Da}. If the building is locally finite, then the Davis realization 
$(|\Delta_D|,d)$ is a proper metric space by Theorem~\ref{MkappaProper}, 
and $\Aut(\cB)$ is locally compact,
totally disconnected, and acts properly and isometrically on $|\Delta_D|$.
\end{Def}
\begin{Prop}\label{DavisTopology}
Let $\cB$ be a building, with chamber set $C$. Let $V$ be the vertex set of the subdivided Davis complex.
The following topologies on $\Aut(\cB)$ coincide.
\begin{enumerate}[\rm(i),nosep]
 \item The topology of pointwise convergence on the discrete set $C$.
 \item The topology of pointwise convergence on the discrete set $V$.
 \item The topology of pointwise convergence, for any topology on $|\Delta_D|$ in which $V$ is a discrete subset.
 \item The compact-open topology, for the weak topology  on $|\Delta_D|$.
\end{enumerate}
Suppose that $\Aut(\cB)$ carries this topology and that $|\Delta_D|$ carries some topology.
Then the joint evaluation map
\[
 \Aut(\cB)\times|\Delta_D|\longrightarrow|\Delta_D|
\]
is continuous if one of the following holds.
\begin{enumerate}[\rm(a),nosep]
 \item The topology on $|\Delta_D|$ is invariant and locally finite.
 \item The topology on $|\Delta_D|$ comes from an invariant metric $d$ in which $V$ is discrete.
\end{enumerate}
\end{Prop}
All the assumptions in Proposition~\ref{DavisTopology} and in Proposition~\ref{ProperOnDelta_D} below are 
met by the CAT(0) metric on the Davis realization $|\Delta_D|$ as an $M_0$-complex.
\begin{proof}
Every chamber is a vertex in  $\Delta_D$, hence the topologies in (i) and (ii) coincide.
The rest of the proof is as in Proposition~\ref{TopologyFromDeltaC}.
\end{proof}
Similarly, we have the following result about proper actions.
\begin{Prop}\label{ProperOnDelta_D}
Suppose that a topological group $G$ acts as a group of automorphisms on a building $\cB$ and that $|\Delta_D|$
carries a locally finite, $G$-invariant topology in which the vertex set $V$ is closed and discrete.
Assume also that $|\Delta_D|\times|\Delta_D|$ is a $k$-space.
The following are equivalent.
\begin{enumerate}[\rm(i),nosep]
 \item The action of $G$ on $V$ is continuous and proper.
 \item The action of $G$ on $|\Delta_D|$ is continuous and proper.
\end{enumerate}
If these hold, then $G$ is locally compact.
\end{Prop}
\begin{Ex}
Let $A,B$ be infinite groups, and put $G=A*B$. The corresponding right-angled building is the Bass-Serre tree $T$
of this free product, and $\Delta_D$ is the barycentric subdivision of $T$. The action of the discrete group $G$ on the chamber set
$C=G$ is free and therefore proper. On the other hand, the vertex stabilizers in the tree are infinite and hence
the action on the vertices is not proper.
\end{Ex}
However, we have the following result, which supplements Proposition~\ref{ProperOnDelta_D}.
\begin{Lem}
Suppose that a topological group $G$ acts as a group of automorphisms on a locally finite building $\cB$.
The following are equivalent.
\begin{enumerate}[\rm(i),nosep]
 \item The action of $G$ on the chamber set $C$ is continuous and proper.
 \item The action of $G$ on the set of spherical residues $V$ is continuous and proper.
\end{enumerate}
\end{Lem}
\begin{proof}
The action is continuous and proper if stabilizers are compact and open.
Every chamber is a spherical residue and therefore (ii) implies (i).
Suppose that (i) holds and that $v=\Res_J(c)$ is a spherical residue.
Then $G_c\subseteq G_v$ and therefore $G_v$ is open.
Since $\cB$ is locally finite, $\Res_J(c)$ is finite and thus $G_c$ has finite index in $G_v$.
Hence $G_v$ is compact.
\end{proof}
We end with a geometric variant of Theorem~\ref{inftrans}.
By an \emph{infinite line} $L$ in a metric space $X$ we mean an isometric embedding 
$L:\RR\longrightarrow X$. Suppose that $|\Delta_D|$ is the Davis realization
of a building $\cB$, with its intrinsic CAT(0) metric $d$. Then for every infinite line
$L$ in $|\Delta_D|$, there is an apartment $A$ in $\cB$ such that 
\[
 L(\RR)\subseteq |A_D|\subseteq|\Delta_D|,
\]
where $|A_D|$ is the Davis realization of $A$ \cite[Thm.~E]{CapHag}. 
Lines exist in the Davis complex, provided that $W$ is infinite.

We say that two infinite lines $L,L'$ in $|\Delta_D|$ are of the \emph{same type}
if there exist apartments $A,A'$ in $\cB$, with $L(\RR)\subseteq |A_D|$
and $L'(\RR)\subseteq |A'_D|$, and a building isomorphism $\psi:A\longrightarrow A'$
with $L'=\psi\circ L$.
\begin{Prop}\label{linetrans}
Let $\cB$ be a building of type $(W,I)$, with $W$ infinite, and let $G$ be a topological group that acts as a group of automorphisms on $\cB$.
If the action on the chambers is proper and Weyl-transitive, then it is transitive on the set of all infinite lines of any given type.
\end{Prop}
\begin{proof}
Let $L,L'$ be two infinite lines of the same type, and let $A,A'$ be apartments in $\cB$ whose
Davis realizations contain $L$ and $L'$, respectively. Let $\psi:A\longrightarrow A'$ be an 
isomorphism with $L'=\psi\circ L$.
For every $t\in \RR$ we choose a maximal simplex
$s_t$ in $A_D$ with $L(t)\in|s_t|$. Since $s_t$ is maximal, it is of the form 
\[
 s_t=\{\{c_t\}\subseteq \Res_{j}(c_t)\subseteq\cdots\subseteq\Res_{J_{m_t}}(c_t)\},
\]
for some chamber $c_t$ and some maximal spherical subset $J_{m_t}\subseteq I$.
We put $c'_t=\psi(c_t)$ and $s_t'=\psi(s_t)$.
For every $t\geq 0$ there is some $g\in G$ with 
\[g(c_t)=c'_t\text{ and } g(c_{-t})=c'_{-t},\]
because the action of $G$ is Weyl-transitive.
Then
\[g(L(t))=L'(t)\text{ and }g(L(-t))=L'(-t),\]
because $g(s_t)=s'_t$ and $g(s_{-t})=s'_{-t}$. We put 
\[
 G_t=\{g\in G\mid g(L(t))=L'(t)\text{ and }g(L(-t))=L'(-t)\}.
\]
Then $G_t$ is compact (because the action is proper) and nonempty 
by the remark above.
For $g\in G_t$ and $-t\leq s\leq t$ we have 
$g(L(s))=L'(s)$, because geodesics in CAT(0) spaces are unique.
Hence $G_{r+t}\subseteq G_t$ for $r\geq 0$.
Thus $\bigcap\{G_t\mid t\geq0\}\neq\emptyset$.
\end{proof}
I do not know if transitivity on the set of infinite lines of any fixed type implies Weyl-transitivity.
This is related to the question which geodesics in the Davis realization are contained in infinite lines.

We finally apply our results to euclidean buildings and we recall the relevant notions.
\begin{Def}\label{EucDef}
Suppose that $\cH$ is a collection of affine hyperplanes in some euclidean space $\RR^m$,
such that $\cH$ is locally finite and invariant under the group $W$ generated by the 
reflections along the hyperplanes in $\cH$.
The hyperplanes in $\cH$ are also called \emph{walls}.
If the action of $W$ on $\RR^m$ is cocompact, then there is a compact convex polytope $C$,
bounded by a finite set of hyperplanes from $\cH$, which is a fundamental domain
for the $W$-action. The reflections $i_1,\ldots,i_r$ along the codimension~$1$ faces of $C$ generate $W$ 
as a Coxeter group, and $(W,I)$ is called a Coxeter system of \emph{euclidean type},
for $I=\{i_1,\ldots,i_r\}$. We refer to \cite[Ch.~VI.1]{Br}.

The Davis realization of the Coxeter system 
in its CAT(0) metric is then  
isometric to the euclidean space $\RR^m$. If the Coxeter system is irreducible, then 
the Tits realization coincides with the Davis realization, and $C$ is a simplex. Otherwise, the Davis realization
is the cartesian product of the Tits (or Davis) realizations of the irreducible components of the Coxeter system,
and $C$ is a cartesian product of simplices (a \emph{complex polysimplicial} in the language of \cite{BT1}).
We call a building $\cB$ of type $(W,I)$ a \emph{euclidean building}.
\end{Def}
The books \cite{AB}, \cite{Br}, \cite{Ro} and \cite{Weiss2} discuss mainly the case of euclidean buildings where the Coxeter system is irreducible.
However, for the results that we consider here irreducibility is not really important.
One just has to keep in mind that the Davis realization is
the cartesian product of the Davis realizations of the irreducible factors, and that the fundamental chamber is no longer a simplex,
but rather a cartesian product of simplices.
Examples of euclidean buildings arise from semisimple algebraic groups over fields with discrete valuations \cite{BT1}.
The CAT(0) Davis realization of a euclidean building is a euclidean building in the sense of Kleiner and Leeb \cite{KL}
(but not conversely). We refer to \cite{Par}, \cite{Da}, and \cite[Ch.~12]{AB}. 
\begin{Thm}\label{ProperEuc}
Let $\cB$ be a euclidean building and let $G$ be a topological group that 
acts as a group of automorphisms on $\cB$. Assume also that the action on the chambers is proper.
The following are equivalent.
\begin{enumerate}[\rm(i),nosep]
  \item The $G$-action is Weyl-transitive.
  \item The $G$-action is strongly transitive on the maximal atlas.
 \end{enumerate}
\end{Thm}
\begin{proof}
We noted before that (ii) implies (i).
Suppose that the action is Weyl-transitive. By Proposition~\ref{linetrans}, it is transitive on all infinite lines
of a fixed type in its Davis realization. Let $A$ be an apartment in $\cB$.
If the infinite line $L$ in $|A_D|$ is not at 
bounded distance from any wall in $|A_D|$, then $L$ is a regular geodesic in the sense of \cite{KL},
and $A$ is the unique apartment containing $L$~\cite[4.6.4]{KL}.
It follows that $G$ acts transitively on the set of all apartments in $\cB$, and hence transitively on the maximal atlas.
The Coxeter group $W$ acts sharply transitively on the chambers in $A$.
For every $w\in W$, the composite $L'=w\circ L$ is an infinite line of the same type as $L$.
It follows that there is an element $g\in G$ with $gL'=L$. This element fixes $A$ (because $A$ is the unique apartment
containing $L$ and $L'$), and $g|_A=w$ because the only element of $W$ that fixes $L$ pointwise is the identity.
This shows that the $G$-stabilizer of $A$ acts transitively on the chambers in
$A$, and hence $G$ acts strongly transitively on the maximal atlas of $\cB$.
\end{proof}
\begin{Def}
Suppose that $(W,I)$ is a Coxeter system of euclidean type acting on $\RR^m$, as in Definition~\ref{EucDef}.
The Coxeter group $W$ decomposes as a semidirect product $W=\ZZ^{m}\rtimes W_0$,
for some spherical (i.e.~finite) Coxeter group $W_0\subseteq W$. 
We may assume that the origin $0$ is the unique fixed point of $W_0$.
The walls passing through $0$ subdivide $\RR^m$ into finitely many infinite cones
which are called \emph{sectors} or \emph{Weyl chambers}.

We call an atlas $\bf A$ for a euclidean building $\cB$ \emph{good} \cite[11.8.4]{AB} if it satisfies the following extra condition:
If $S_1\subseteq A_1$ and $S_2\subseteq A_2$ are sectors in apartments $A_1,A_2\in\bf A$, then there
is an apartment $A\in \bf A$ and sectors $S_1',S_2'\subseteq A$, with 
$S_1'\subseteq S_1$ and $S_2'\subseteq S_2$. In this situation, there is a spherical building 
$\partial_{\bf A}\cB$ attached to $(\cB,\bA)$, the \emph{spherical building at infinity}
\cite[11.8.4]{AB} \cite[Ch.~8]{Weiss2}. The chambers of this spherical building correspond to equivalence classes of sectors in apartments in $\bf A$, where two
sectors are called equivalent if their intersection contains a sector. The apartments of
$\partial_{\bf A}\cB$ are in one-to-one correspondence with the apartments in $\bf A$ \cite[Prop.~8.27]{Weiss2}.
The maximal atlas of $\cB$ is a good atlas \cite[7.24]{Weiss2}, and in this case the spherical building at infinity
can be identified with the Tits boundary of the Davis realization in its CAT(0) metric.
\end{Def}
We apply Theorem~\ref{ProperEuc} to Bruhat-Tits buildings, as defined in \cite[Def.~13.1]{Weiss2}.
\begin{Def}\label{BTB}
A \emph{Bruhat-Tits building} (a \emph{Bruhat-Tits pair} in the terminology of \cite{Weiss2})
consists of an irreducible thick euclidean building $\cB$ and a good atlas $\bA$,
such the spherical building at infinity $\partial_\bA\cB$ is a Moufang building  (as defined in \cite[29.15]{Weiss2}).
Note that from the definition of a Moufang building, the dimension of $\cB$ is at least 2.
The classification of these buildings is carried out in detail in \cite{Weiss2}.
We remark that every thick irreducible euclidean building of rank at least $4$ (simplicial dimension at least $3$)
is automatically a Bruhat-Tits building, for any good atlas \cite[p.~119]{Weiss2}.
We refer also to the classification \cite{GKVW}, and to \cite{CapCio} and \cite{KS} for conditions implying strong transitivity.
\end{Def}
\begin{Thm}\label{Arthur}
Let $(\cB,\bA)$ be a locally finite Bruhat-Tits building, and let $G^\dagger\subseteq\Aut(\cB)$ denote the group generated by
all root groups of the spherical building $\partial_\bA\cB$. Then the closure $H$ of $G^\dagger$ (in the topology of pointwise convergence on chambers)
is locally compact, second countable, totally disconnected, acts properly on the chambers, and strongly transitively on the maximal atlas of $\cB$.
\end{Thm}
The following example may be instructive.
\begin{Ex}
Let $p$ be a prime number, let \[\nu:\mathbb{Q}\longrightarrow \ZZ\cup\{\infty\}\] denote the $p$-adic valuation
and let \[A=\{a\in\mathbb Q\mid \nu(a)\geq 0\}\subseteq\mathbb Q\] denote the corresponding valuation ring.
Then $\mathbb Q^m$ is an $A$-module.
A \emph{lattice} $L\subseteq\mathbb Q^m$ is an $A$-submodule which is free of rank~$m$. Two such lattices $L,L'\subseteq\mathbb{Q}^m$ are equivalent if there
is $q\in\mathbb Q^\times$ with $qL=L'$.
Let $\cB$ denote the euclidean building whose vertices are the equivalence classes of lattices \cite[V.8]{Br}.
The group $G=\mathrm{PSL}_m(\mathbb Q)$ acts strongly transitively on $\cB$, with respect to a good atlas
$\bA$ whose apartments correspond to $m$-tuples of points in general position in the projective space
$\mathbb{Q}P^{m-1}$. If $m\geq 3$, then $(\cB,\bA)$ is a Bruhat-Tits building in the sense of Definition~\ref{BTB}. The building $\cB$ is locally finite and 
hence proper, but the group $G=G^\dagger$ is not locally compact. However, it acts Weyl-transitively on $\cB$.
The closure of $G$ in the automorphism group is the group 
$\mathrm{PSL}_m(\mathbb{Q}_p)$. 
\end{Ex}
\begin{proof}[Proof of Theorem~\ref{Arthur}]
By \cite[Thm. 12.31]{Weiss2}, the group $G^\dagger$ can be considered a subgroup
of $\Aut(\cB)$. 
Since the building at infinity $\partial_\bA\cB$ is Moufang, the group $G^\dagger$ acts transitively on the set of its apartments. 
By \cite[8.27]{Weiss2}, the group $G^\dagger$
acts transitively on the atlas $\bA$. Let $A$ be an apartment of $\cB$ 
and let $W'$ denote the subgroup of $\Aut(A)$ induced by the stabilizer of $A$ 
in $G^\dagger$.  By \cite[Prop. 13.5 and Prop. 13.28]{Weiss2}, 
$W'$ contains the Weyl group $W$ of $A$. In particular,
$W'$ acts transitively on the chambers of $A$. Thus $G^\dagger$ acts
strongly transitively on $\bA$. In particular, $G^\dagger$ acts Weyl-transitively on $\cB$.

Since $\Aut(\cB)$ is closed in the isometry group $\Iso(|\Delta_D|)$ of the Davis realization by 
Theorem~\ref{AutBuilding}, the group $H$ is contained in $\Aut(\cB)$.
Moreover, the group $H$ acts properly on $\cB$.
Hence $H$ acts strongly transitively on the maximal atlas by Theorem~\ref{ProperEuc}.

The group $\Iso(|\Delta_D|)$ is locally compact and second countable by Theorem~\ref{IsoPropThm},
and $\Aut(\cB)$ is totally disconnected by Theorem~\ref{AutBuilding}.
\end{proof}

\subsubsection*{Acknowledgments}
Arthur Bartels' request for a solid reference for Theorem~\ref{Arthur} initiated a first, three page version of the present article.
I thank Bertrand R\'emy and Guy Rousseau for some remarks on Bruhat-Tits buildings.
Stefan Witzel, Herbert Abels and Polychronis Strantzalos spotted mistakes and pointed out some more references.
In particular, the authors kindly informed me that the forthcoming book \cite{AS} will contain the results that we present in Section~2.
Richard Weiss made helpful comments on the proof of Theorem~\ref{Arthur} and the material surrounding it.
Philip M\"oller read the article, spotted several mistakes and made many good suggestions.
I would also like the referee, who raised several very good questions about an 
earlier version of the manuscript.

\end{document}